\documentclass{birkau}
\usepackage{amsmath,amssymb,amsthm,mathtools,url}
\usepackage{hyperref}
\hypersetup{hidelinks}
\usepackage[T1]{fontenc}
\usepackage{lmodern}
\usepackage{microtype}
\usepackage{mathrsfs}
\usepackage{enumitem}
\usepackage[nameinlink,capitalize]{cleveref}
\usepackage{aliascnt}
\usepackage{xcolor}

\DeclareMathOperator{\Aut}{Aut}
\DeclareMathOperator{\GL}{GL}
\DeclareMathOperator{\Spec}{Spec}
\DeclareMathOperator{\CE}{CE}
\DeclareMathOperator{\CCE}{CCE}
\DeclareMathOperator{\Idem}{Idem}
\DeclareMathOperator{\Stab}{Stab}
\DeclareMathOperator{\id}{id}
\DeclareMathOperator{\Sym}{Sym}
\newtheorem{theorem}{Theorem}[section]
\newaliascnt{proposition}{theorem}
\newtheorem{proposition}[proposition]{Proposition}
\aliascntresetthe{proposition}
\newaliascnt{lemma}{theorem}
\newtheorem{lemma}[lemma]{Lemma}
\aliascntresetthe{lemma}
\newaliascnt{corollary}{theorem}
\newtheorem{corollary}[corollary]{Corollary}
\aliascntresetthe{corollary}

\theoremstyle{definition}
\newtheorem{definition}[theorem]{Definition}
\newtheorem{example}[theorem]{Example}

\newtheorem{remark}[theorem]{Remark}

\Crefname{theorem}{Theorem}{Theorems}
\Crefname{proposition}{Proposition}{Propositions}
\Crefname{lemma}{Lemma}{Lemmas}
\Crefname{corollary}{Corollary}{Corollaries}
\Crefname{definition}{Definition}{Definitions}
\Crefname{example}{Example}{Examples}
\Crefname{remark}{Remark}{Remarks}
\Crefname{question}{Question}{Questions}

\title[Rigidity and factor reconstruction for monoids with zero]
{Direct-product rigidity and factor \\ reconstruction for monoids with zero}

\author{Joseph Atalaye}
\address{Department of Mathematical Sciences, Stellenbosch University,
Stellenbosch 7600, South Africa}
\email{26828146@sun.ac.za}

\corrauthor{Liam Baker}
\address{Department of Mathematical Sciences, Stellenbosch University,
Stellenbosch 7600, South Africa}
\email{liambaker@sun.ac.za}
\urladdr{https://math.sun.ac.za/liambaker}

\author{Sophie Marques}
\address{Centro de Matem\'atica, Department of Mathematics, University of Minho,
Campus de Gualtar, 4710--057 Braga, Portugal}
\email{smarques@math.uminho.pt}
\urladdr{https://sites.google.com/site/sophiemarques64/}

\begin{document}


\begin{abstract}
We study direct products of monoids with zero and give a criterion under which their factors can be recovered from the multiplicative structure alone.
While classical decomposition theory encodes direct products through factor congruences, central elements, and refinement properties, we give a concrete multiplicative reconstruction mechanism.
If the factors have no nontrivial complemented central idempotents, then the coordinate idempotents are precisely the atoms and coatoms of the complemented-central-idempotent poset, and their multiplicative stabilizers are precisely the coordinate factors and cofactors.
It follows that every isomorphism between such products is monomial, yielding the corresponding wreath-product description of automorphism groups.
More generally, every product decomposition is obtained by grouping the original factors; in particular, strict refinement follows.
We apply these results to multiplicative monoids of directly indecomposable unital rings, including connected commutative rings, and to arithmetic examples arising from residue-class rings.
\end{abstract}

\maketitle

\section{Introduction}
\label{sec:introduction}

Automorphisms of direct products can be considerably more complicated than products of automorphisms of the individual factors.
In general, homomorphisms between different factors may create off-diagonal terms and allow the coordinates to mix.
A basic rigidity problem is therefore to determine when the direct factors are forced by the multiplicative structure itself, so that every product isomorphism is factorwise up to permutation.

For groups, automorphisms of direct products have often been studied through matrices of homomorphisms between the factors; see \cite{BidwellCurranMcCaughan,BidwellII} and the determinant approach of \cite{BresciaDeterminant}.
Related questions for semigroups and monoids appear in \cite{AhmadidelirDoostie}.
In \cite{AtalayeBakerMarquesProducts}, we used a matrix method to study products of multiplicative monoids arising from commutative rings.

There is also a classical decomposition theory in universal algebra, where direct products are encoded by factor congruences and refinement properties control compatibility and uniqueness of decompositions; see \cite{ChangJonssonTarski,HoefnagelSRP}.
Boolean factor congruences are classical as well, and central elements have been developed as abstract decomposition data generalizing, among other examples, central idempotents in rings; see
\cite{SanchezTerrafBFC,SanchezTerrafVaggione, WillardBFC}.
Thus neither the general philosophy that central data encode product decompositions nor the theory of strict refinement is new here.

The contribution of the present paper is a different and very concrete multiplicative reconstruction mechanism for monoids with an absorbing zero.
Under a natural condition on complemented central idempotents, the individual coordinate factors can be recovered directly from multiplication.
The key distinction is between recognizing the coordinate directions and recovering the corresponding factors.
Centrality provides an isomorphism-invariant semilattice in which the relevant coordinate idempotents can be recognized, while their multiplicative stabilizers recover the factors themselves. 
To our knowledge, this coatom--stabilizer reconstruction has not previously been used to recover direct factors of monoids in this form.
It is this multiplicative detector, rather than the classical existence of central decomposition data, that is the distinctive mechanism of the paper.

This immediately leads to rigidity.
Since an isomorphism preserves complemented central idempotents, their order, and their stabilizers, it must permute the distinguished coordinate idempotents and hence the coordinate factors.
We prove that every isomorphism between products whose factors have no nontrivial complemented central idempotents is monomial: it is obtained by permuting the factors and applying isomorphisms factor by factor.
Once such monomial rigidity is known, the wreath-product description of automorphism groups is a standard consequence; we record it explicitly, including the case of repeated isomorphic factors.

The same reconstruction mechanism controls much more than automorphisms.
The strict refinement property itself is classical, and there are several general frameworks and sufficient conditions for it; see \cite{ChangJonssonTarski,IskanderSRP,HoefnagelSRP}.
What is more specific here is the concrete description of every product decomposition: each one is obtained by grouping the original coordinate factors, and the common refinement of two decompositions is obtained by intersecting the corresponding blocks.
Strict refinement is therefore a consequence of this explicit multiplicative reconstruction, not an additional hypothesis and not the novelty claim by itself.

The hypothesis is especially natural for multiplicative monoids of rings.
The classical ring-theoretic fact is that central idempotents encode direct-product decompositions of a unital ring.
In particular, a unital ring is directly indecomposable precisely when it has no nontrivial central idempotents.
Our application is different in that the maps under consideration are only isomorphisms of multiplicative monoids: no preservation of addition is assumed.
Thus the usual ring decomposition theorem does not by itself imply our rigidity statement.
This should also be distinguished from unique-addition results, where one proves that multiplication determines the additive structure; see \cite{StephensonUA}.
Our argument requires no unique-addition hypothesis.
In the commutative case, direct indecomposability is equivalent to connectedness of the spectrum, so the method applies well beyond domains and local rings.

These results also strengthen the framework of
\cite{AtalayeBakerMarquesProducts}.
That earlier work obtains factorwise automorphisms for products of multiplicative monoids of $D$-rings under total-ring-of-fractions and pairwise-cardinality hypotheses.
The present reconstruction replaces those product hypotheses by the single structural requirement that each factor have no nontrivial complemented central idempotents; for unital rings, this is equivalent to direct indecomposability.
It also accommodates repeated isomorphic factors through the permutation part of the monomial theorem.
Together with the prime-power calculations of
\cite{AtalayeBakerMarquesPrimePower}, this yields explicit local-to-global descriptions for the multiplicative automorphism groups of residue-class rings.

The novelty claim of the paper can therefore be stated narrowly:
classical central and refinement theories provide the surrounding decomposition framework, while our new ingredient is the purely multiplicative passage.

This mechanism yields the monomial rigidity theorem and the explicit classification of product decompositions from which strict refinement follows.

\medskip
\noindent\textbf{Organization of the paper.}
Section~\ref{sec:central} develops the complemented-central-idempotent machinery and the factor-reconstruction argument.
Section~\ref{sec:rigidity} proves the rigidity theorem, the wreath-product description of automorphism groups, and the refinement results.
Section~\ref{sec:rings} gives the applications to rings and arithmetic examples.

\section{Complemented central idempotents and factor reconstruction}
\label{sec:central}

We introduce the multiplicative structure that will be used to recover the factors of a direct product.
Central idempotents as carriers of decomposition data are classical, both for rings and in their universal-algebraic generalizations by central elements and factor congruences; see, for example, \cite{WillardBFC,SanchezTerrafVaggione}.
The point of this section is not to rederive that general theory.
Instead, for monoids with zero we isolate the complemented central idempotents and show that their order, combined with multiplicative stabilizers, directly reconstructs the coordinate factors.
The relevant objects are therefore not all central idempotents, but those admitting complements in the natural order on central idempotents.

\begin{definition}
Let $M$ be a monoid.
An element $e\in M$ is a \emph{central idempotent} if
$
e^2 = e
$
and
$
ex = xe
$
for every $x\in M$.
We write
$$
\CE(M)
= 
\{e\in M:e^2 = e\text{ and }ex = xe\text{ for all }x\in M\}.
$$
\end{definition}

Because central idempotents commute, their product is again a central idempotent.
Thus $\CE(M)$ is a commutative idempotent monoid under multiplication and hence a meet-semilattice for the order
$$
e\leq f
\quad\Longleftrightarrow\quad
ef = e.
$$
When $M$ has a zero, $0$ and $1$ are respectively the least and greatest elements.
Centrality is essential at this point: arbitrary idempotents need not commute, their product need not be idempotent, and they do not in general form such a semilattice under multiplication.

\begin{definition}
Let $M$ be a monoid with zero.
A central idempotent $e\in\CE(M)$ is \emph{complemented} if there exists $f\in\CE(M)$ such that
$
ef = 0
$
and $1$ is the unique common upper bound of $e$ and $f$ in $\CE(M)$.
We denote the set of complemented central idempotents by
$
\CCE(M).
$
\end{definition}

\begin{remark}
The definition involves only multiplication, the absorbing $0$, and the identity $1$, and does not require $\CE(M)$ to be a lattice.
The elements $0$ and $1$ are always complementary, so
$
\{0,1\}\subseteq\CCE(M).
$
\end{remark}

The restriction to complemented central idempotents is important.
A nontrivial central idempotent of a monoid need not reflect a nontrivial direct-product decomposition.
Recall that a nontrivial monoid $M$ is \emph{directly indecomposable} if every decomposition
$
M\cong A\times B
$
has a trivial factor.

\begin{example}
\label{ex:three-chain}
Let
$
L_3 = \{0<e<1\}
$
be the three-element chain with multiplication given by meet,
$
xy = x\wedge y.
$
Then $L_3$ is a commutative monoid with zero and every element is a central idempotent.
Nevertheless, $L_3$ is directly indecomposable:
if
$
L_3\cong A\times B
$
with $A$ and $B$ nontrivial, then
$
3 = |A||B|,
$
which is impossible.

The middle element $e$ is not complemented.
Hence
$
\CCE(L_3) = \{0,1\},
$
whereas
$
\CE(L_3) = L_3.
$
Thus the condition
$
\CCE(M) = \{0,1\}
$
is strictly weaker than
$
\CE(M) = \{0,1\}.
$
\end{example}

We next record the behaviour of central idempotents under products.
For the remainder of the paper, let $I$ and $J$ be arbitrary index sets, possibly infinite.

\begin{lemma}
\label{lem:ce-product}
Let $(M_i)_{i \in I}$ be a tuple of monoids.
Then
$$
\CE\!\left(\prod_{i \in I} M_i\right)
= 
\prod_{i \in I} \CE(M_i).
$$
\end{lemma}

\begin{proof}
Both idempotence and centrality in a direct product are coordinatewise conditions, so the result follows immediately.
\end{proof}

\begin{lemma}
\label{lem:cce-product}
Let the $M_i$ be monoids with zero for $i \in I$.
Then
$$
\CCE\!\left(\prod_{i \in I} M_i\right)
= 
\prod_{i \in I} \CCE(M_i).
$$
\end{lemma}

\begin{proof}
By \cref{lem:ce-product}, $\CE(\prod_i M_i) = \prod_i\CE(M_i)$ with the coordinatewise order, and both the condition $ef = 0$ and the uniqueness of $1$ as a common upper bound are coordinatewise.
Hence $e = (e_i)$ is complemented if and only if each $e_i$ is complemented.
\end{proof}

We now turn to factor reconstruction.
The first observation is independent of any condition on complemented central idempotents.

Let
$$
M = \prod_{i \in I} M_i
$$
be a product of nontrivial monoids with zero.
For $i \in I$, we define the complements
$$
\varepsilon_i
= 
(1,\dotsc,1,\underset{i}{0},1,\dotsc,1)
= 
\left(\begin{cases}
    1   & j \neq i \\
    0   & j = i
\end{cases}\right)_{j \in I}
,\quad
\delta_i
=
\left(\begin{cases}
    0   & j \neq i \\
    1   & j = i
\end{cases}\right)_{j \in I}
$$
and let
$
\iota_i : M_i \longrightarrow M
$
and 
$
\pi_i : M \longrightarrow M_i
$
be the standard coordinate embedding and projection,
$$
\iota_i(x)
= 
(1,\dotsc,1,\underset{i}{x},1,\dotsc,1)
\qquad\text{and}\qquad
\pi_i(x) = x_i.
$$
\begin{definition}
Let $M$ be a monoid and $m\in M$.
The \emph{(right) multiplicative stabilizer} of $m$ is
$$
\Stab_M(m) = \{x\in M:mx = m\}.
$$
When \(m\) is central, the left and right multiplicative stabilizers coincide.
\end{definition}

\begin{lemma}
\label{lem:stabilizer-submonoid}
For every $m\in M$, the set $\Stab_M(m)$ is a submonoid of $M$.
\end{lemma}

\begin{proof}
Clearly $1\in\Stab_M(m)$.
If $x,y\in\Stab_M(m)$, then
$$
 m(xy) = (mx)y = my = m,
$$
so $xy\in\Stab_M(m)$.
\end{proof}

Stabilizers behave coordinatewise under direct products:

\begin{lemma}
\label{lem:stabilizer-product}
Let
$
M = \prod_{i \in I} M_i$ and let $m = (m_i)_{i \in I} \in M$.
Then
$$
\Stab_M(m)
= 
\prod_{i \in I} \Stab_{M_i}(m_i).
$$
\end{lemma}

\begin{proof}
For $x = (x_i)_{i \in I} \in M$, one has
$
mx = m
$
if and only if
$
m_ix_i = m_i
$
for every $i$, which is equivalent to
$
x_i\in\Stab_{M_i}(m_i)
$
for every $i$.
\end{proof}

We now apply this observation to the coordinate idempotents.

\begin{proposition}
\label{prop:stabilizer-factor}
For every $i \in I$,
$$
\Stab_M(\varepsilon_i)
= 
\iota_i(M_i)
\cong
M_i
\qquad\text{and}\qquad
\Stab_M(\delta_i)
=
\pi_i^{-1}(\{1\}).
$$
\end{proposition}

\begin{proof}
By \cref{lem:stabilizer-product},
$$
\Stab_M(\varepsilon_i)
= 
\prod_{j \in I} \Stab_{M_j}((\varepsilon_i)_j).
$$
Since
$
\Stab_{M_i}(0) = M_i
$
and
$
\Stab_{M_j}(1) = \{1\}
$,
we obtain
\[
\Stab_M(\varepsilon_i)
= 
\prod_{j \in I} \left(\begin{cases}
    \{1\}   & j \neq i \\
    M_i     & j = i
\end{cases}\right)
= 
\iota_i(M_i)
\cong
M_i.
\]
The other equality is proved similarly.
\end{proof}

Conversely, the stabilizer determines the corresponding coordinate idempotent.

\begin{lemma}
\label{lem:stabilizer-detects-epsilon}
Let $e \in M$ be an idempotent.
\begin{enumerate}
    \item If
    $
    \Stab_M(e) = \iota_i(M_i)
    $
    for some $i$, then
    $
    e = \varepsilon_i.
    $
    \item If
    $
    \Stab_M(e) = \pi_i^{-1}(\{1\})
    $
    for some $i$, then
    $
    e = \delta_i.
    $
\end{enumerate}
\end{lemma}

\begin{proof}
Let $e = (e_i)_{i \in I} \in M$.
By \cref{lem:stabilizer-product},
$$
\Stab_M(e)
= 
\prod_{j\in I} \Stab_{M_j}(e_j).
$$
\begin{enumerate}
    \item 
    Thus
    $$
    \Stab_{M_i}(e_i) = M_i
    \qquad\text{and}\qquad
    \Stab_{M_j}(e_j) = \{1\}
    \ \text{for}\ j\neq i.
    $$
    Since $0\in\Stab_{M_i}(e_i)$, one has
    $
    e_i0 = e_i,
    $
    and therefore $e_i = 0$.
    For $j\neq i$, idempotence gives
    $
    e_j\in\Stab_{M_j}(e_j) = \{1\},
    $
    so $e_j = 1$.
    Hence
    $
    e = \varepsilon_i.
    $
    \item This is proved similarly. \qedhere
\end{enumerate}
\end{proof}

The preceding results show that the coordinate factors and cofactors are always stabilizers of the corresponding coordinate idempotents.
The remaining question is when those idempotents can themselves be recognized from the monoid.
The condition on complemented central idempotents used below gives an exact answer.

Once $\varepsilon_i$ is known, its stabilizer recovers the $i$-th factor purely multiplicatively.
The role of centrality is instead to make the coordinate idempotents recognizable without reference to the chosen product presentation: central idempotents form an isomorphism-invariant meet-semilattice, and the complemented part of this semilattice singles out the coordinate directions under the hypothesis below.
Thus the zero--identity pair supplies the coordinate idempotents and their stabilizers recover the corresponding factors,
whereas centrality identifies the correct idempotents among the idempotents of the monoid.
The use of central data to encode product decompositions is classical; the specific coatom--stabilizer detector established below is the new multiplicative step.

Recall that a \emph{coatom} of a bounded partially ordered set is a maximal element strictly below its greatest element, and that an \emph{atom} of a bounded partially ordered set is a minimal element strictly above its least element.

\begin{proposition}
\label{prop:factor-reconstruction-characterization}
Let
$
M = \prod_{i\in I} M_i
$
be a product of nontrivial monoids with zero.
The following conditions are equivalent:
\begin{enumerate}[label = \textup{(\roman*)}]
\item
$
\CCE(M_i) = \{0,1\}
$
for every $i$;

\item the coatoms of $\CCE(M)$ are precisely
$
\varepsilon_i
$
for $i \in I$;

\item the coordinate factors are precisely the stabilizers of the coatoms of $\CCE(M)$, i.e.,
$$
\left\{
\iota_i(M_i) : i \in I
\right\}
= 
\left\{
\Stab_M(e):e\text{ is a coatom of }\CCE(M)
\right\}.
$$

\item the atoms of $\CCE(M)$ are precisely $\delta_i$ for $i \in I$;

\item the coordinate cofactors are precisely the stabilizers of the atoms of $\CCE(M)$, i.e.,
$$
\left\{
\pi_i^{-1}(\{1\}) : i \in I
\right\}
= 
\left\{
\Stab_M(e) : e\text{ is an atom of }\CCE(M)
\right\}.
$$
\end{enumerate}
\end{proposition}

\begin{proof}
Assume \textup{(i)}.
By \cref{lem:cce-product},
$$
\CCE(M)
= 
\prod_{i \in I} \CCE(M_i)
= 
\{0,1\}^I.
$$
With the coordinatewise order, its coatoms are exactly the elements having a single zero coordinate, namely $\varepsilon_i$ for $i \in I$, and its atoms are exactly the elements having a single nonzero coordinate, namely $\delta_i$ for $i \in I$.
Thus \textup{(i)} implies \textup{(ii)} and \textup{(iv)}.

Assume \textup{(ii)}.
By \cref{prop:stabilizer-factor},
$
\Stab_M(\varepsilon_i) = \iota_i(M_i)
$
for every $i$.
Since the $\varepsilon_i$ are precisely the coatoms of $\CCE(M)$, condition \textup{(iii)} follows.

Assume \textup{(iii)}.
For every $i$, there is a coatom $e\in\CCE(M)$ such that
$
\Stab_M(e) = \iota_i(M_i)
$.
By \cref{lem:stabilizer-detects-epsilon}, necessarily
$
e = \varepsilon_i.
$
Hence every $\varepsilon_i$ is a coatom of $\CCE(M)$.

If some $M_i$ admitted an element
$
c\in\CCE(M_i)\setminus\{0,1\},
$
then \cref{lem:cce-product} would give
$
u = \iota_i(c) \in \CCE(M)
$.
Since $0<c<1$, one would have
$
\varepsilon_i<u<1,
$
contradicting the fact that $\varepsilon_i$ is a coatom.
Therefore
$
\CCE(M_i) = \{0,1\}
$
for every $i$, proving \textup{(i)} and completing the cycle \textup{(i)} $\implies$ \textup{(ii)} $\implies$ \textup{(iii)} $\implies$ \textup{(i)}.

The proof of \textup{(iv)} $\implies$ \textup{(v)} $\implies$ \textup{(i)} is similar: the implication \textup{(iv)} $\implies$ \textup{(v)} is exactly parallel to \textup{(ii)} $\implies$ \textup{(iii)}, using \cref{prop:stabilizer-factor}, while for \textup{(v)} $\implies$ \textup{(i)} we (similarly to \textup{(iii)} $\implies$ \textup{(i)}) show that every $\delta_i$ is an atom of $\CCE(M)$, and that if some $M_i$ admitted an element
$
c\in\CCE(M_i)\setminus\{0,1\}
$,
then the element
$$
u = \left(\begin{cases}
    0   & j \neq i  \\
    c   & j = i
\end{cases}\right)_{j\in I} \in \CCE(M)
$$
would satisfy
$
0 < u < \delta_i,
$
contradicting the fact that $\delta_i$ is an atom.
\end{proof}

\begin{remark}
The condition
$
\CCE(M_i) = \{0,1\}
$
is therefore precisely what makes the coordinate factors and cofactors recoverable by the coatom--stabilizer and atom-stabilizer constructions respectively:
the coordinate idempotents are exactly the coatoms and atoms of $\CCE(M)$, and their stabilizers are exactly the coordinate factors and cofactors.
This characterizes the reconstruction mechanism used in the rigidity theorem; it is not a characterization of all possible uniqueness phenomena for direct-product decompositions.
\end{remark}

We finish the section by recording that every ingredient in this reconstruction is preserved by isomorphisms.

\begin{lemma}
\label{lem:isomorphism-central}
Let $M$ and $N$ be monoids with zero and let
$
\Phi:M\xrightarrow{\sim}N
$
be a monoid isomorphism.
Then $\Phi$ restricts to an isomorphism of bounded partially ordered sets
$$
\Phi|_{\CE(M)}:
\CE(M)\xrightarrow{\sim}\CE(N),
$$
and
$$
\Phi(\CCE(M)) = \CCE(N).
$$
In particular, $\Phi$ induces an order isomorphism
$$
\CCE(M)\xrightarrow{\sim}\CCE(N)
$$
and therefore preserves coatoms and atoms.
\end{lemma}

\begin{proof}
A monoid isomorphism preserves $1$ and idempotents.
It also preserves the absorbing zero.
Indeed, for $y = \Phi(x)\in N$,
$$
\Phi(0)y
= 
\Phi(0x)
= 
\Phi(0),
\qquad
y\Phi(0)
= 
\Phi(x0)
= 
\Phi(0),
$$
so $\Phi(0)$ is the zero of $N$.

If $e\in\CE(M)$ and $y = \Phi(x)\in N$, then
$$
\Phi(e)y
= 
\Phi(ex)
= 
\Phi(xe)
= 
y\Phi(e),
$$
so $\Phi(e)\in\CE(N)$.
Applying the same argument to $\Phi^{-1}$ gives
$$
\Phi(\CE(M)) = \CE(N).
$$
Moreover,
$$
e\leq f
\iff
ef = e
\iff
\Phi(e)\Phi(f) = \Phi(e)
\iff
\Phi(e)\leq\Phi(f).
$$
Thus $\Phi$ induces an order isomorphism on central idempotents.
Since it preserves $0$, $1$, and common upper bounds, it preserves complementarity.
Hence
$$
\Phi(\CCE(M)) = \CCE(N).
$$
The final assertion follows because order isomorphisms preserve coatoms and atoms.
\end{proof}

\begin{lemma}
\label{lem:stabilizers-functorial}
Let
$
\Phi:M\xrightarrow{\sim}N
$
be a monoid isomorphism and let $m\in M$.
Then
$$
\Phi(\Stab_M(m))
= 
\Stab_N(\Phi(m)).
$$
\end{lemma}

\begin{proof}
For $x\in M$,
\begin{align*}
x\in\Stab_M(m)
&\iff
mx = m
\\
\iff
\Phi(m)\Phi(x) = \Phi(m)
&\iff
\Phi(x)\in\Stab_N(\Phi(m)).
\end{align*}
Since $\Phi$ is bijective, the result follows.
\end{proof}

Under the equivalent conditions of \cref{prop:factor-reconstruction-characterization}, the coordinate factors and cofactors are therefore determined from the product by an isomorphism-invariant construction: they are precisely the stabilizers of the coatoms and atoms of $\CCE(M)$ respectively.
Consequently, any isomorphism between two such products must permute these distinguished submonoids.
This is the structural input for the rigidity theorem of the next section.
\section{Rigidity and refinement of direct products}
\label{sec:rigidity}

We now apply the factor-reconstruction criterion of \cref{prop:factor-reconstruction-characterization}.
Abstract refinement theory already gives classical routes from suitable refinement hypotheses to uniqueness of irreducible direct-product factorizations; see \cite{ChangJonssonTarski,HoefnagelSRP}.
The point here is that we do not assume such a refinement property.
Instead, when the factors have no nontrivial complemented central idempotents, the multiplication itself determines the coordinate factors as the stabilizers of the coatoms of $\CCE$.
Since both coatoms and stabilizers are preserved by isomorphisms, product isomorphisms are forced to preserve the factor structure.

The resulting isomorphisms are the direct-product analogue of monomial matrices.

\begin{definition}
Let
$
M = \prod_{i \in I} M_i
$
and
$
N = \prod_{j \in J} N_j
$.
An isomorphism
$
\Phi : M \xrightarrow{\sim} N
$
is called \emph{monomial} if there exists a bijection
$
\sigma : I \to J
$
and isomorphisms
$
\phi_i : M_i\xrightarrow{\sim}N_{\sigma(i)}
$
such that
$
\Phi\circ\iota_i
= 
\iota'_{\sigma(i)}\circ\phi_i
$
and
$
\pi_{\sigma(i)}' \circ \Phi
=
\phi_i \circ \pi_i
$
for every $i \in I$,
where the $\iota_j'$ and $\pi_j'$ for $j \in J$ are the canonical embeddings and projections for $N$.
\end{definition}

\begin{remark}
If $I$ (and thus $J$) is finite, then having
$
\Phi\circ\iota_i
= 
\iota'_{\sigma(i)}\circ\phi_i
$
for every $i \in I$ is sufficient, since every \(x\in M\) admits the factorization
\[
x=\prod_{i\in I}\iota_i(\pi_i(x)).
\]
For infinite \(I\), this factorization is not available, and the projection identities must be established separately.
\end{remark}

We can now state the main rigidity theorem.

\begin{theorem}
\label{thm:rigidity}
Let
$
M = \prod_{i \in I} M_i
$
and
$
N = \prod_{j \in J} N_j
$
where all $M_i$ and $N_j$ are nontrivial monoids with zero satisfying
$
\CCE(M_i) = \CCE(N_j) = \{0,1\}.
$
Then every monoid isomorphism
$
\Phi : M\xrightarrow{\sim}N
$
is monomial.
More precisely, there exists a unique bijection
$
\sigma : I \to J
$
and unique isomorphisms
$$
\phi_i:M_i\xrightarrow{\sim}N_{\sigma(i)}
$$
such that
$$
\Phi\circ\iota_i
= 
\iota'_{\sigma(i)}\circ\phi_i
\qquad\text{and}\qquad
\pi_{\sigma(i)}' \circ \Phi
=
\phi_i \circ \pi_i
$$
for every $i$.
\end{theorem}

\begin{proof}
By \cref{prop:factor-reconstruction-characterization}, the atoms and coatoms of $\CCE(M)$ are precisely $\delta_i$ and $\varepsilon_i$ for $i \in I$ and those of $\CCE(N)$ are precisely $\delta_j'$ and $\varepsilon'_j$ for $j \in J$.
By \cref{lem:isomorphism-central}, $\Phi$ induces an order isomorphism
$$
\CCE(M)\xrightarrow{\sim}\CCE(N),
$$
and therefore a bijection between their coatoms, and thus a bijection $\sigma : I \to J$ such that
$$
\Phi(\varepsilon_i) = \varepsilon'_{\sigma(i)}
$$
for every $i \in I$.
Since $\delta_i$ is the complement of $\varepsilon_i$ and $\Phi$ preserves the order, the absorbing $0$, and the identity $1$, the element $\Phi(\delta_i)$ is a complement of \( \Phi(\varepsilon_i) \).
But
\[
\CCE(N)=\{0,1\}^{J},
\]
in which the complement of $\varepsilon'_{\sigma(i)}$ is uniquely $\delta'_{\sigma(i)}$.
Hence
\[
\Phi(\delta_i)=\delta'_{\sigma(i)}.
\]
By \cref{lem:stabilizers-functorial},
$$
\Phi\bigl(\Stab_M(\varepsilon_i)\bigr)
= 
\Stab_N\bigl(\varepsilon'_{\sigma(i)}\bigr)
\qquad\text{and}\qquad
\Phi\bigl(\Stab_M(\delta_i)\bigr)
=
\Stab_N\bigl(\delta_{\sigma(i)}'\bigr).
$$
Using \cref{prop:stabilizer-factor}, this becomes
$$
\Phi\bigl(\iota_i(M_i)\bigr)
= 
\iota'_{\sigma(i)}(N_{\sigma(i)})
\qquad\text{and}\qquad
\Phi\bigl(\pi_i^{-1}(\{1\})\bigr)
=
\pi_{\sigma(i)}'^{-1}(\{1\}).
$$
Thus
$$
\phi_i
= 
(\iota'_{\sigma(i)})^{-1}\circ\Phi\circ\iota_i
:
M_i\xrightarrow{\sim}N_{\sigma(i)}
$$
is a well defined isomorphism.
Both $\sigma$ and the $\phi_i$ are uniquely determined by $\Phi$;
in fact
$$
\iota_{\sigma(i)}' \circ \phi_i
=
\Phi \circ \iota_i
\implies
\phi_i
=
\pi_{\sigma(i)}' \circ \iota_{\sigma(i)}' \circ \phi_i
=
\pi_{\sigma(i)}' \circ \Phi \circ \iota_i
$$
for each $i \in I$.

Now for \(x=(x_j)_{j\in I}\in M\), coordinatewise multiplication gives
\[
x\delta_i
=
\delta_i\iota_i(x_i).
\]
Applying \(\Phi\), we obtain
\[
\Phi(x)\delta'_{\sigma(i)}
=
\Phi(x\delta_i)
=
\Phi\bigl(\delta_i\iota_i(x_i)\bigr)
=
\delta'_{\sigma(i)} \iota'_{\sigma(i)}\bigl(\phi_i(x_i)\bigr).
\]
Taking the \(\sigma(i)\)-th coordinate gives
\[
\pi'_{\sigma(i)}(\Phi(x))
=
\phi_i(x_i)
=
\phi_i(\pi_i(x)).
\]
Hence
\(
\pi'_{\sigma(i)}\circ\Phi
=
\phi_i\circ\pi_i
\)
for every \(i\in I\), as required.
\end{proof}

\begin{remark}
\label{rem:unique-product}
In particular, among products satisfying the hypotheses of \cref{thm:rigidity}, the family of isomorphism types of the coordinate factors is determined up to bijective reindexing by the product.More precisely, if
\[
\prod_{i\in I}M_i
\cong
\prod_{j\in J}N_j
\]
and all factors satisfy the hypotheses of \cref{thm:rigidity}, then there exists a bijection
$
\sigma:I\xrightarrow{\sim}J
$
such that
$
M_i\cong N_{\sigma(i)}
$
for every $i\in I$.
\end{remark}

 The following wreath-product formula is a standard consequence once monomial rigidity has been established; it is included to make the automorphism group explicit, especially when isomorphic factors occur with multiplicity. 


\begin{corollary}
\label{cor:wreath}
Suppose
\[
M\cong\prod_{\lambda\in\Lambda}P_\lambda^{I_\lambda},
\]
where the $P_\lambda$ are pairwise nonisomorphic nontrivial monoids with zero satisfying
$
\CCE(P_\lambda)
=
\{0,1\}
$,
and each $I_\lambda$ is a nonempty set.
Then
\[
\Aut(M)
\cong
\prod_{\lambda\in\Lambda}
\left(
\Aut(P_\lambda)^{I_\lambda}
\rtimes
\Sym(I_\lambda)
\right),
\]
where $\Sym(I_\lambda)$ acts by permuting the coordinates.
\end{corollary}
\begin{proof}
Write the coordinate set of \(M\) as the disjoint union $\bigsqcup_{\lambda\in\Lambda} I_\lambda$, where the coordinates indexed by \(I_\lambda\) all have factor $P_\lambda$.
By \cref{thm:rigidity}, every automorphism of \(M\) induces a bijection of the coordinate set and acts on each coordinate by an isomorphism of the corresponding factors.
Since the \(P_\lambda\) are pairwise nonisomorphic, a coordinate indexed by \(I_\lambda\) can only be sent to another coordinate indexed by \(I_\lambda\).
Hence the induced bijection restricts, for every \(\lambda\), to a permutation $\sigma_\lambda\in\Sym(I_\lambda)$.
Moreover, for each \(i\in I_\lambda\), the corresponding factor isomorphism is an automorphism $\phi_{\lambda,i}\in\Aut(P_\lambda)$.
Thus an automorphism of \(M\) is uniquely determined by the family
$$
\bigl((\phi_{\lambda,i})_{i\in I_\lambda},
      \sigma_\lambda\bigr)_{\lambda\in\Lambda}.
$$
Conversely, any such family defines an automorphism of \(M\) by permuting the coordinates in each \(I_\lambda\) according to \(\sigma_\lambda\) and applying \(\phi_{\lambda,i}\) coordinatewise.
Composition gives the natural semidirect-product structure, with \(\Sym(I_\lambda)\) acting on \(\Aut(P_\lambda)^{I_\lambda}\) by permuting the coordinates.
Therefore
\[
\Aut(M)
\cong
\prod_{\lambda\in\Lambda}
\left(
\Aut(P_\lambda)^{I_\lambda}
\rtimes
\Sym(I_\lambda)
\right).
\qedhere
\]
\end{proof}

\begin{remark}
\label{rem:special-wreath-cases}
Two useful special cases are worth recording.
If the $M_i$ are pairwise nonisomorphic, then
$$
\Aut\left(\prod_{i\in I} M_i\right)
\cong
\prod_{i \in I} \Aut(M_i).
$$
At the opposite extreme, for a nontrivial monoid $P$ with zero satisfying
$
\CCE(P) = \{0,1\},
$
one has
$$
\Aut(P^I)
\cong
\Aut(P)^I \rtimes \Sym(I),
$$
where $\Sym(I)$ denotes the symmetric group permuting the elements of $I$.
\end{remark}

\begin{example}
\label{ex:free-monoid-zero}
Let $X$ be a nonempty set, $X^*$ be the free monoid on $X$, and
$$
(X^*)^0 = X^*\sqcup\{0\}
$$
be obtained by adjoining an absorbing zero.
The only idempotent of $X^*$ is the identity: if $w^2 = w$, then comparison of word lengths gives
$$
2|w| = |w|,
$$
and hence $w = 1$.
Therefore
$$
\CE((X^*)^0) = \CCE((X^*)^0) = \{0,1\}.
$$
Consequently,
$$
\Aut\!\left(((X^*)^0)^n\right)
\cong
\Aut((X^*)^0)^n \rtimes S_n,
$$
where $S_n$ denotes the symmetric group on $n$ elements.

If $X$ is finite of cardinality $d$, every automorphism of $X^*$ permutes the free generators:
$$
\Aut((X^*)^0)\cong S_d.
$$
Thus
$$
\Aut\!\left(((X^*)^0)^n\right)
\cong
(S_d)^n \rtimes S_n.
$$
\end{example}

\subsection{Relation with strict refinement}
\label{subsec:strict-refinement}

The factor-reconstruction argument also controls arbitrary product decompositions.
This places the preceding rigidity theorem in the classical theory of refinement of direct products.
Strict refinement itself is not new: it goes back to classical refinement theory, and it has been studied through Boolean factor congruences, factorable congruences, and categorical coextensivity; see
\cite{ChangJonssonTarski,WillardBFC,IskanderSRP,HoefnagelSRP}.
Our contribution in this subsection is the stronger concrete 
\cref{prop:canonical-refinement}: for the monoids considered here, every product decomposition is explicitly obtained by grouping the original factors.
The strict-refinement corollary then follows by intersecting the corresponding blocks.

We recall the relevant notion in the form needed here.
A family of homomorphisms
$$
(f_i : M\longrightarrow A_i)_{i \in I}
$$
is a \emph{product decomposition} if the induced map
$$
(f_i)_{i \in I} : M \xrightarrow{\sim} \prod_{i \in I} A_i
$$
is an isomorphism.
\begin{definition}
\label{def:finite-strict-refinement}
A monoid $M$ has the \emph{strict refinement property} if, for every pair of product decompositions
$$
(f_i : M\longrightarrow A_i)_{i \in I},
\qquad
(g_j : M\longrightarrow B_j)_{j\in J},
$$
there exist monoids $C_{ij}$ and homomorphisms
$$
\alpha_{ij} : A_i\longrightarrow C_{ij},
\qquad
\beta_{ij} : B_j\longrightarrow C_{ij}
$$
such that
$$
\alpha_{ij} f_i = \beta_{ij} g_j
$$
for every $(i,j)\in I\times J$, and such that the induced maps
$$
A_i\longrightarrow\prod_{j\in J} C_{ij},
\qquad
B_j\longrightarrow\prod_{i\in I} C_{ij}
$$
are isomorphisms for every $i$ and $j$.
\end{definition}

In other words, $M$ has the strict refinement property if any two product decompositions admit a common rectangular refinement.

Let again
$
M = \prod_{i \in I} M_i.
$
For $S \subseteq I$, put
$$
M_S = \prod_{i\in S} M_i,
$$
with $M_\varnothing$ the trivial monoid, let
$$
\pi_S : M\longrightarrow M_S
\qquad\text{and}\qquad
\iota_S : M_S\longrightarrow M
$$
be the coordinate projection and embedding respectively, and let
$$
\varepsilon_S
=
(e_{i,S})_{i \in I},
\ 
e_{i,S}
=
\begin{cases}
0,  & i\in S \\
1,  & i\notin S
\end{cases}
\quad\text{and}\quad
\delta_S
=
(d_{i,S})_{i \in I},
\ 
d_{i,S}
=
\begin{cases}
1,  & i\in S \\
0,  & i\notin S
\end{cases}.
$$
Note that $\varepsilon_S$ and $\delta_S$ are complemented for each $S \subseteq I$.
By \cref{lem:stabilizer-product},
$$
\Stab_M(\varepsilon_S)
=
\iota_S(M_S)
\qquad\text{and}\qquad
\Stab_M(\delta_S)
=
\pi_S^{-1}(\{1\}).
$$

The following result gives a concrete description of all product decompositions.

\begin{proposition}
\label{prop:canonical-refinement}
Let
$$
M = \prod_{i \in I} M_i,
$$
where the $M_i$ are nontrivial monoids with zero satisfying
$
\CCE(M_i) = \{0,1\}
$
for each $i$.
Let
$$
F = (f_j)_{j \in J} :
M\xrightarrow{\sim}
\prod_{j\in J} A_j
$$
be a product decomposition with the $A_j$ nontrivial.
Then there is a unique partition
$$
I
= 
\bigsqcup_{j\in J} S_j
$$
and isomorphisms
$$
\theta_j : A_j \xrightarrow{\sim} M_{S_j}
$$
such that
$$
\theta_j f_j = \pi_{S_j}
$$
for every $j \in J$.

Thus every product decomposition of $M$ is obtained by grouping the original coordinate factors.
\end{proposition}

\begin{proof}

We denote the product $\prod_{j \in J} A_j$ by $A$.
Since $M$ has an absorbing zero and $F$ is an isomorphism, $A$ has absorbing zero $F(0)$.
Write
$
F(0) = (0_j)_{j\in J}.
$
For every $j\in J$, the element $0_j$ is an absorbing zero of $A_j$.
Indeed, for $x\in A_j$, multiplying $F(0)$ on either side by the tuple which is $x$ in the $j$-th coordinate and $1$ elsewhere gives
$
0_j x = 0_j = x 0_j.
$
Thus every $A_j$ has an absorbing zero.

For $j \in J$, let
$$
\epsilon_j
=
(1,\dotsc,1,\underset{j}{0_j},1,\dotsc,1)
\in
A
\qquad\text{and}\qquad
\delta_j
=
(0,\dotsc,0,\underset{j}{1_j},0,\dotsc,0)
\in
A.
$$
Since $\epsilon_j$ is a complemented central idempotent,
$$
e_j \coloneqq F^{-1}(\epsilon_j)
$$
belongs to $\CCE(M)$.
By \cref{lem:cce-product},
$
\CCE(M) = \{0,1\}^I,
$
so there is a unique subset
$
S_j \subseteq I
$
such that
$
e_j = \varepsilon_{S_j}
$.

Let
$$
\jmath_j : A_j \longrightarrow A
\qquad\text{and}\qquad
\varpi_j : A \to A_j
$$
be the coordinate embedding and projection respectively.
Since
$$
\jmath_j(A_j) = \Stab(\epsilon_j),
$$
functoriality of stabilizers (\cref{lem:stabilizers-functorial}) gives
$$
F^{-1}\bigl(\jmath_j(A_j)\bigr)
= 
\Stab_M(e_j)
= 
\Stab_M(\varepsilon_{S_j})
= 
\iota_{S_j}(M_{S_j}).
$$
Hence
$
F^{-1}\jmath_j
$
induces an isomorphism
$$
\theta_j : A_j \xrightarrow{\sim} M_{S_j}
$$
such that
$
\iota_{S_j} \circ \theta_j = F^{-1} \circ \jmath_j
$,
or equivalently
$$
\theta_j = \pi_{S_j} \circ F^{-1} \circ \jmath_j
$$
since $\pi_{S_j} \circ \iota_{S_j}$ is the identity.

Note that $F^{-1}(\delta_j) = \delta_{S_j}$ is the complement of $\varepsilon_{S_j}$ in $\{0,1\}^I$, and hence is $1$ precisely on the coordinates in $S_j$.
Now we can write the identity morphism pointwise as
$
\id_A = g_j \cdot (\jmath_j \circ \varpi_j)
$
where $g_j : A \to A$ sends the $j$th coordinate to $1$, i.e. by \cref{lem:stabilizers-functorial},
\begin{align*}
    g_j(A)
    &=
    \varpi_j^{-1}(\{1\})
    =
    \Stab_A(\delta_j)
    \\
    \iff
    F^{-1}(g_j(A))
    &=
    \Stab_M(F^{-1}(\delta_j))
    =
    \Stab_M(\delta_{S_j})
    =
    \pi_{S_j}^{-1}(\{1\})
    \\
    \iff
    \pi_{S_j} \circ F^{-1} \circ g_j (A)
    &=
    \{1\}.
\end{align*}
Thus pointwise
\begin{align*}
    \pi_{S_j} \circ F^{-1}
    &=
    (\pi_{S_j} \circ F^{-1} \circ g_j) \cdot (\pi_{S_j} \circ F^{-1} \circ \jmath_j \circ \varpi_j) = 1 \cdot (\theta_j \circ \varpi_j)
    \\
    \pi_{S_j}
    &=
    \theta_j \circ \varpi_j \circ F = \theta_j \circ f_j
\end{align*}
as desired.
Since $A_j$ is nontrivial, $S_j$ is nonempty.

It remains to show that the $S_j$ form a partition.
Now
$
\delta_j\delta_k = 0
$
for $j \neq k$, so that
$
S_j \cap S_k = \varnothing.
$
Moreover, for every $i \in I$ we have that
$
F(\iota_i(0)) \neq 1_A,
$
hence there is a $j \in J$ such that by \cref{lem:stabilizers-functorial},
\begin{align*}
    \varpi_j(F(\iota_i(0))) &\neq 1 \\
    \iff F(\iota_i(0)) &\notin \varpi_j^{-1}(\{1\}) = \Stab_A(\delta_j) \\
    \iff \iota_i(0) &\notin F^{-1}(\Stab_A(\delta_j)) = \Stab_M(F^{-1}(\delta_j)) = \Stab_M(\delta_{S_j}) \\
    \iff \iota_i(0) \delta_{S_j} &\neq \delta_{S_j} \\
    \iff i &\in S_j.
\end{align*}
Since this is true for all $i \in I$, we get that
$$
\bigcup_{j \in J} S_j = I.
$$
Thus the $S_j$ form a partition.

The uniqueness of $S_j$ follows from the uniqueness of
$
e_j = F^{-1}(\epsilon_j).
$ 
Once $S_j$ is fixed, $\theta_j$ is uniquely determined by $\theta_j f_j = \pi_{S_j}$ since $f_j$ is surjective.
\end{proof}

\begin{corollary}
\label{cor:strict-refinement}
Let
$$
M = \prod_{i \in I} M_i,
$$
where the $M_i$ are nontrivial monoids with zero satisfying
$$
\CCE(M_i) = \{0,1\}.
$$
Then $M$ has the strict refinement property.
\end{corollary}

\begin{proof}
After deleting trivial factors, let
$$
(f_j : M\longrightarrow A_j)_{j\in J},
\qquad
(g_k : M\longrightarrow B_k)_{k\in K}
$$
be two product decompositions.
By
\cref{prop:canonical-refinement}, they correspond to partitions
$$
I
= 
\bigsqcup_{j \in J} S_j
= 
\bigsqcup_{k \in K} T_k
$$
and identifications
$$
A_j \cong M_{S_j},
\qquad
B_k \cong M_{T_k},
$$
compatible with the coordinate projections.

For $(j,k) \in J \times K$, set
$$
C_{jk} = M_{S_j \cap T_k}.
$$
The intersections $S_j \cap T_k$, as $k$ varies, partition $S_j$,
while, as $j$ varies, they partition $T_k$.
Hence
$$
A_j \cong \prod_{k \in K} C_{jk},
\qquad
B_k \cong \prod_{j \in J} C_{jk}.
$$
Taking the corresponding coordinate projections gives maps $\alpha_{jk}$ and $\beta_{jk}$ satisfying
$$
\alpha_{jk} f_j
= 
\pi_{S_j \cap T_k}
= 
\beta_{jk} g_k.
$$
Thus the two decompositions admit a common rectangular refinement.
\end{proof}

\begin{remark}
The preceding proposition is the concrete contribution of this subsection.
The existence and general theory of strict refinements are classical; what is specific here is that the multiplicative coatom--stabilizer construction identifies every decomposition with a literal partition of the original coordinate set.
Consequently the common refinement of two decompositions is explicitly obtained by intersecting the corresponding blocks.
The strict refinement property is therefore a consequence of the same factor-reconstruction mechanism used to prove rigidity, rather than an additional hypothesis.
\end{remark}

\subsection{Why the absorbing zero matters}
\label{subsec:zero}

The absorbing zero is essential to the reconstruction mechanism.
Without it, even the complete absence of nontrivial idempotents does not prevent automorphisms from mixing direct factors.

\begin{example}
\label{ex:groups-no-zero}
Let $p$ be a prime.
The group
$
C_p\times C_p
$
has no idempotents other than its identity, but
$$
\Aut(C_p\times C_p)
\cong
\GL_2(\mathbb F_p).
$$
In additive notation, for example,
$$
(x,y)\longmapsto(x+y,y)
$$
is an automorphism which mixes the two cyclic factors and is not monomial with respect to the displayed decomposition.

Thus trivial idempotent structure alone does not imply product rigidity.
In the setting of \cref{thm:rigidity}, the absorbing zero produces the coordinate idempotents
$
\varepsilon_i,
$
whose stabilizers recover the factors.
\end{example}

Adjoining an absorbing zero restores precisely this mechanism.

\begin{example}
Let $G$ be a group and let
$
G^0 = G\sqcup\{0\}
$
be the monoid obtained by adjoining an absorbing zero.
Since a group has no idempotent other than its identity,
$$
\CE(G^0) = \CCE(G^0) = \{0,1\}.
$$
Therefore
$$
\Aut((G^0)^n)
\cong
\Aut(G^0)^n \rtimes S_n
=
\Aut(G)^n \rtimes S_n.
$$
\end{example}

\begin{remark}
The identity plays a role complementary to that of the absorbing zero.
Together, $0$ and $1$ produce the coordinate idempotents $\varepsilon_i$ and $\delta_i$.
The identity also allows the factors to be embedded as
\[
\iota_i(M_i)
=
\{(x_j)_{j\in I}:x_i\in M_i,\ x_j=1\text{ for }j\neq i\}.
\]
For every $x=(x_j)_{j\in I}\in M$, we have
\(
x\delta_i=\delta_i\iota_i(x_i)
\),
so multiplication by $\delta_i$ isolates the $i$-th coordinate.
When $I$ is finite, one may even reconstruct $x$ from
\[
x = \prod_{i\in I}\iota_i(x_i),
\]
but this finite factorization is neither available nor needed for arbitrary products.

For a semigroup with zero but without identity, the coordinate axes obtained by inserting zeros remain subsemigroups, but their products do not reconstruct arbitrary elements of the direct product.
Thus the coatom--stabilizer method developed here does not extend formally to arbitrary semigroups without identity; at that level, product decompositions are more naturally encoded by factor congruences.
\end{remark}

\section{Consequences for rings and arithmetic}
\label{sec:rings}

We now apply the preceding results to multiplicative monoids of rings.
The relation between central idempotents and direct-product decompositions of unital rings is classical; the first two propositions below are recalled only to translate our monoid hypothesis into standard ring-theoretic language.
The new point for the present paper is the consequence for isomorphisms of the multiplicative monoids, where no additivity is assumed.

For a unital ring $R$, the central idempotents of the multiplicative monoid $(R,\cdot)$ are exactly the central idempotents of $R$ in the usual ring-theoretic sense.
Moreover, every such idempotent is automatically complemented.

\begin{proposition}
\label{prop:ring-cce}
Let $R$ be a unital ring.
Then
$$
\CCE(R,\cdot) = \CE(R,\cdot).
$$
More precisely, if $e$ is a central idempotent, then $1-e$ is its complement in $\CE(R,\cdot)$.
Consequently,
$$
\CCE(R,\cdot) = \{0,1\}
\quad\Longleftrightarrow\quad
R\text{ has no nontrivial central idempotents}.
$$
\end{proposition}

\begin{proof}
Let $e$ be a central idempotent.
Then $1-e$ is again a central idempotent and
$
e(1-e) = 0
$.
Suppose that $g\in\CE(R,\cdot)$ is a common upper bound of $e$ and $1-e$.
Then
$
eg = e
$
and
$
(1-e)g = 1-e
$.
Adding the two equalities gives
$$
g = (e+1-e)g = e+(1-e) = 1.
$$
Thus $1-e$ is a complement of $e$, and every central idempotent is complemented.
\end{proof}

The preceding proposition identifies the hypothesis of the monoid rigidity theorem with direct indecomposability on the ring side.

\begin{proposition}
\label{prop:ring-direct-indecomposable}
Let $R$ be a nonzero unital ring.
The following conditions are equivalent:
\begin{enumerate}[label = \textup{(\roman*)}]
\item $R$ is directly indecomposable as a ring;
\item $R$ has no nontrivial central idempotents;
\item $\CE(R,\cdot) = \{0,1\}$;
\item $\CCE(R,\cdot) = \{0,1\}$.
\end{enumerate}
\end{proposition}

\begin{proof}
The equivalence of \textup{(ii)}, \textup{(iii)}, and \textup{(iv)}
follows from \cref{prop:ring-cce}.

If $e$ is a nontrivial central idempotent, then
$$
R\longrightarrow eR\times(1-e)R,
\qquad
r\longmapsto (er,(1-e)r),
$$
is an isomorphism of rings, with inverse
$
(x,y)\mapsto x+y.
$
Thus a directly indecomposable ring has no nontrivial central idempotents.

Conversely, if
$
R\cong R_1\times R_2
$
with $R_1$ and $R_2$ nonzero, then the inverse image of $(1,0)$ is a central idempotent distinct from $0$ and $1$.
Hence \textup{(ii)}
implies \textup{(i)}.
\end{proof}

We therefore obtain the following multiplicative rigidity statement.
It should not be confused with the classical uniqueness of ring-product decompositions: the isomorphism below is not assumed to preserve addition.
Nor do we assume that the factors are unique-addition rings in the sense of \cite{StephensonUA}; the factor reconstruction is performed directly inside the multiplicative monoid.

\begin{corollary}
\label{cor:ring-rigidity}
Let
$R_i$ for $i \in I$
and
$S_j$ for $j \in J$
be nonzero directly indecomposable unital rings.
Then every multiplicative-monoid isomorphism
$$
\prod_{i \in I} (R_i,\cdot)
\xrightarrow{\sim}
\prod_{j \in J} (S_j,\cdot)
$$
is monomial.
In particular, $|I| = |J|$, and after a unique bijection of the factors the isomorphism is induced by multiplicative-monoid isomorphisms
$$
(R_i,\cdot)\xrightarrow{\sim}(S_{\sigma(i)},\cdot).
$$
\end{corollary}

\begin{proof}
By \cref{prop:ring-direct-indecomposable},
$$
\CCE(R_i,\cdot) = \CCE(S_j,\cdot) = \{0,1\}
$$
for all $i$ and $j$.
The result follows from
\cref{thm:rigidity}.
\end{proof}

The noncommutative case illustrates particularly clearly why the relevant objects are central idempotents rather than arbitrary idempotents.

\begin{example}
\label{ex:matrix-rings}
Let $k$ be a field and $d\geq2$.
The matrix ring $M_d(k)$ contains many nontrivial idempotents, for instance
$
\operatorname{diag}(1,0,\dotsc,0).
$
Such idempotents do not in general encode direct-product decompositions.
Indeed, if
$$
e = \operatorname{diag}(1,0,\dotsc,0)
\qquad\text{and}\qquad
f = I_d-e,
$$
then $ef = fe = 0$, and any idempotent $g$ satisfying
$
eg = e$ and $fg = f$ must be $I_d$, since adding the two equalities gives $g = (e+f)g = e+f = I_d$.
Thus even relations formally resembling complementarity may occur for an idempotent which does not define a product factor.
The missing condition is centrality.

On the other hand,
$
Z(M_d(k)) = kI_d
$.
Hence a central idempotent has the form $\lambda I_d$ with
$
\lambda^2 = \lambda,
$
so necessarily $\lambda = 0$ or $1$.
Therefore
$$
\CE(M_d(k),\cdot)
= 
\CCE(M_d(k),\cdot)
= 
\{0,I_d\}.
$$
Thus $M_d(k)$ is directly indecomposable, and
\cref{cor:ring-rigidity} applies to arbitrary products of matrix rings.
This illustrates why centrality is essential in the reconstruction criterion: arbitrary idempotents may describe corners or retract-like structures, while central idempotents are the ones compatible with genuine product splittings.

For example, every multiplicative-monoid isomorphism between two arbitrary products of matrix rings
$$
\prod_i M_{d_i}(k_i)
\qquad\text{and}\qquad
\prod_j M_{e_j}(\ell_j)
$$
must permute the multiplicative monoids of the matrix-ring factors and act factorwise.
\end{example}

\begin{remark}
In the commutative case, direct indecomposability is equivalent to connectedness of the spectrum.
Thus \cref{cor:ring-rigidity} applies to all connected commutative rings, not only to domains or local rings.
In particular, neither the presence of zero divisors nor the existence of several maximal ideals creates additional mixing between factors.
\end{remark}

\subsection{Application to \texorpdfstring{$D$}{D}-rings}

Recall that a commutative ring $R$ is a $D$-ring if every zero divisor is nilpotent.
Such a ring has no nontrivial idempotents.
Indeed, if
$
e^2 = e
$
with
$
e\neq0,1,
$
then
$
e(1-e) = 0
$
and $1-e\neq0$, so $e$ is a zero divisor.
Hence $e$ is nilpotent.
Since an idempotent which is nilpotent must be zero, this is a contradiction.
Thus
$
\Idem(R) = \{0,1\}
$.
Equivalently, $\Spec R$ is connected, and therefore
$$
\CCE(R,\cdot) = \{0,1\}.
$$

Consequently, the preceding rigidity results apply to arbitrary commutative $D$-rings.
Relative to \cite{AtalayeBakerMarquesProducts}, the product-rigidity input is therefore strengthened in two ways: no total-ring-of-fractions hypothesis is required, and repeated isomorphic multiplicative monoids can be retained and handled by the permutation part of \cref{cor:wreath} rather than being excluded by a cardinality condition.

\begin{corollary}
\label{cor:D-ring-product}
Let $R_i$ for $i \in I$ be nonzero commutative $D$-rings such that the multiplicative monoids
$
(R_i,\cdot)
$
are pairwise nonisomorphic.
Then
$$
\Aut\!\left(\prod_{i \in I} (R_i,\cdot)\right)
\cong
\prod_{i \in I} \Aut(R_i,\cdot).
$$
\end{corollary}

\begin{proof}
Each $R_i$ has connected spectrum, and hence
$
\CCE(R_i,\cdot) = \{0,1\}
$.
Since the multiplicative monoids are pairwise nonisomorphic, no nontrivial permutation of factors is possible.
The result follows from \cref{cor:wreath}.
\end{proof}

More generally, if isomorphic multiplicative monoids occur with multiplicities, \cref{cor:wreath} gives the corresponding wreath-product factors.
Thus repeated factors require no separate argument.

\subsection{The multiplicative monoid of \texorpdfstring{$\mathbb Z/n\mathbb Z$}{Z/nZ}}

We finish with the multiplicative monoid of a residue-class ring.
The Chinese-remainder decomposition and the local prime-power calculations are classical or come from our earlier work; the role of the present paper is to supply the general multiplicative rigidity principle that passes from the local factors to the global automorphism group.
Let
$$
n = \prod_{i = 1}^r p_i^{e_i}
$$
be the prime-power decomposition of $n$.
By the Chinese remainder theorem,
$
\mathbb Z/n\mathbb Z
\cong
\prod_{i = 1}^r\mathbb Z/p_i^{e_i}\mathbb Z
$
as rings, and therefore
$$
(\mathbb Z/n\mathbb Z,\cdot)
\cong
\prod_{i = 1}^r
(\mathbb Z/p_i^{e_i}\mathbb Z,\cdot).
$$

Each ring
$
\mathbb Z/p_i^{e_i}\mathbb Z
$
is local, hence directly indecomposable.
Equivalently,
$$
\CCE(\mathbb Z/p_i^{e_i}\mathbb Z,\cdot) = \{0,1\}.
$$
Moreover, the prime powers $p_i^{e_i}$ are pairwise distinct, so the corresponding multiplicative monoids have different cardinalities and are therefore pairwise nonisomorphic.
The wreath-product formula thus reduces to a direct product.

\begin{corollary}
\label{cor:Zn-product}
Let the positive integer $n$ have the prime-power decomposition
$$
n = \prod_{i = 1}^r p_i^{e_i}.
$$
Then
$$
\Aut(\mathbb Z/n\mathbb Z,\cdot)
\cong
\prod_{i = 1}^r
\Aut(\mathbb Z/p_i^{e_i}\mathbb Z,\cdot).
$$
\end{corollary}

\begin{proof}
Apply \cref{cor:wreath} to the decomposition above.
\end{proof}

The local factors were determined in
\cite{AtalayeBakerMarquesPrimePower}.
Combining that calculation with
\cref{cor:Zn-product} therefore gives an explicit local-to-global description of
$
\Aut(\mathbb Z/n\mathbb Z,\cdot)
$
for arbitrary $n$.

\section*{Declarations}

\subsection*{Ethical approval}
Not applicable.

\subsection*{Competing interests} 
Not applicable.

\subsection*{Authors' contributions} 
All authors contributed equally.

\subsection*{Availability of data and materials}
Not applicable. 

\subsection*{Funding}
Not applicable. 

\providecommand{\doi}[1]{DOI~\href{https://doi.org/#1}{\texttt{#1}}}
\providecommand{\arxiv}[1]{arXiV~\href{https://arxiv.org/abs/#1}{\texttt{#1}}}

\end{document}